\documentclass[12pt, twoside, leqno]{article}

\usepackage{amsthm,amsmath,amssymb,mathrsfs,amsfonts,graphicx,graphics,latexsym,exscale,cmmib57,dsfont,amscd}

\usepackage{enumerate}

\usepackage{graphicx}

\newtheorem{thm}{Theorem}[section]
\newtheorem{cor}[thm]{Corollary}
\newtheorem{lem}[thm]{Lemma}

\theoremstyle{definition}
\newtheorem{defn}[thm]{Definition}

\newcommand{\e}{\boldsymbol{o}}
\numberwithin{equation}{section}

\begin{document}


\baselineskip=17pt


\title{Szemer\'{e}di-type theorems for subsets of locally compact abelian groups of positive upper Banach density}

\author{Xiongping Dai\\
Department of Mathematics\\
Nanjing University, Nanjing 210093\\
People's Republic of China\\
E-mail: xpdai@nju.edu.cn
\and
Hailan Liang\\
Department of Mathematics\\
Fuzhou University, Fuzhou 350003\\
People's Republic of China\\
\and
Xinjia Tang\\
Department of Mathematics\\
Nanjing University, Nanjing 210093\\
People's Republic of China}

\date{}

\maketitle

\renewcommand{\thefootnote}{}



\renewcommand{\thefootnote}{\arabic{footnote}}
\setcounter{footnote}{0}


\begin{abstract}
By using ergodic theoretic techniques following Hillel F\"{u}rstenberg, we prove that measurable subsets of a locally compact abelian group of positive upper density contain Szemer\'{e}di-wise configurations defined by an arbitrary compact subset of the group.
\end{abstract}

\section{Introduction}

E.~Szemer\'{e}di's theorem proved in 1975 says that if a set $E\subseteq\mathbb{Z}$ is of positive upper Banach density, i.e.,
\begin{gather*}
\mathrm{BD}^*(E):=\limsup_{N-M\to\infty}\frac{|E\cap[M,N)|}{N-M}>0,
\end{gather*}
then $E$ contains arbitrarily long arithmetic progressions (cf.~\cite{Sze, Fur}). By using their multiple recurrence theorem, H.~F\"{u}rstenberg and Y.~Katznelson in 1978 proved a multidimensional version of Szemer\'{e}di's theorem: If $E\subseteq\mathbb{Z}^n$ is of positive upper Banach density and $F$ is a finite subset of $\mathbb{Z}^n$, then for some vector $u\in\mathbb{Z}^n$ and integer $d\ge1$, $u+dF\subset E$ (cf.~\cite{FK} or \cite[Theorem~7.16]{Fur}).

Based on the above multidimensional version, moreover, Hillel F\"{u}rstenberg proved that if $E\subseteq\mathbb{R}^n$ is of positive upper Banach density with respect to the Lebesgue measure on $\mathbb{R}^n$ and $F$ is a finite subset of $\mathbb{R}^n$, then for some vector $u\in\mathbb{R}^n$ and integer $d\ge1$, $u+dF\subset E$ (cf.~\cite[Theorem~7.17]{Fur}).

In the more general case that $\mathbb{R}^n$ is replaced by an abelian discrete additive group $(G,+)$, using the Stone-\v{C}ech compactification $\beta G$ of $G$ and F\"{u}rstenberg's multiple recurrence theorem, Hindman and Strauss in 2006 proved that if $E\subset G$ is of positive upper density relative to some F{\o}lner net in $G$, then for any $a\in G$ and $l\in\mathbb{N}$, $\{u\in G\,|\,u+d\{a,2a,\dotsc,la\}\subseteq E\}$ has positive upper density relative to the same F{\o}lner net for some $d\in\mathbb{N}$ (cf.~\cite[Theorems~5.5 and 5.6]{HS06}).

In this note, we will extend the usual notion of upper Banach density in $\S\ref{sec1}$, and then by using ergodic theoretic techniques following Hille F\"{u}rstenberg, we shall derive more generalizations of Szemer\'{e}di's theorem for any finite configurations not limited to the form $\{a,2a,\dotsc,la\}$ (cf.~Theorems~\ref{thm3.1} and \ref{thm3.3} and Corollary~\ref{cor3.5} in $\S\ref{sec2}$), for any locally compact Hausdorff, not necessarily discrete, abelian group $G$ with any fixed Haar measure.
\section{F{\o}lner sequences and upper density}\label{sec1}
In this section we will introduce some basic notions and preliminary lemmas to state and prove our main theorems in the next section.

\subsection{Basic notions}
Let $(G,+)$ be a locally compact Hausdorff additive topological group.
According to Haar's theorem (cf.~\cite[Theorem~29C]{Loo}), there exists a left-invariant Haar measure on $(G,+)$, which we denote by $|\centerdot|$ or $dg$.
When $G$ is discrete, then $|\centerdot|$ is just the usual counting measure on $G$.

A sequence of compact subsets $(F_n)_{n=1}^\infty$ in $G$ is called a classical F{\o}lner sequence in $(G,+,|\centerdot|)$ if and only if
$$
\lim_{n\to\infty}\frac{|(g+F_n)\vartriangle F_n|}{|F_n|}=0\quad \forall g\in G.
$$
It is a well-known fact that if $G$ is a locally compact $\sigma$-compact Hausdorff abelian group like $(\mathbb{Z}^m,+)$ with the discrete topology and $(\mathbb{R}^m,+)$ with the Euclidean metric topology, then it has classical F{\o}lner sequences. Although a discrete uncountable abelian group does not have any classical F{\o}lner sequences, yet it always has F{\o}lner nets (cf.~\cite{AW,HS06}).

Since an uncountable abelian group has no classical F{\o}lner sequence under the discrete topology, we need to introduce a nonclassical F{\o}lner sequence in any locally compact Hausdorff group $(G,+,|\centerdot|)$.

\begin{defn}\label{def2.1}
Given any subset $F\subseteq G$, a sequence of compact subsets of $G$, $\mathcal{F}=(F_n)_{n=1}^\infty$, is called an \textit{$F$-F{\o}lner sequence} in $(G,+,|\centerdot|)$ if and only if
\begin{gather}\label{eq1.1}
\lim_{n\to\infty}\frac{|(g+F_n)\vartriangle F_n|}{|F_n|}=0\quad\forall g\in F.
\end{gather}
Notice here that (\ref{eq1.1}) holds only for $g\in F$ but not for any $g\in G$.
\end{defn}

A classical F{\o}lner sequence in $(G,+,|\centerdot|)$ is just a $G$-F{\o}lner sequence. Let us consider an example in order to show an $F$-F{\o}lner sequence in $(G,+,|\centerdot|)$ is not necessarily a classical F{\o}lner sequence for $F\not=G$.
Let $(G,+)=(\mathbb{R}^2,+)$ with $|\centerdot|=dxdy$ under the standard Euclidean topology and $\varepsilon>0$. Define a sequence of thin rectangles $F_n=\{(x,y)\colon|x-n|\le\varepsilon, |y|\le n\}$. Clearly, $|F_n|=4n\varepsilon\to\infty$ as $n\to\infty$. It is easy to see that for $F=\{0\}\times\mathbb{R}$, $(F_n)_{n=1}^\infty$ is an $F$-F{\o}lner sequence but not a classical F{\o}lner sequence in $(G,+,|\centerdot|)$.

\begin{defn}\label{def2.2}
Given any $F\subseteq G$, for any $F$-F{\o}lner sequence $\mathcal{F}=(F_n)_{n=1}^\infty$ in $(G,+,|\centerdot|)$ and any $|\centerdot|$-measurable subset $E\subseteq G$, we set
\begin{gather}
\mathrm{D}_\mathcal{F}^*(E)=\limsup_{n\to\infty}\frac{|E\cap F_n|}{|F_n|},
\end{gather}
which is called the \textit{upper density of $E$ corresponding to $\mathcal{F}$} in $(G,+,|\centerdot|)$.
\end{defn}

We shall say that $E$ is of \textit{positive upper Banach density} in $(G,+,|\centerdot|)$, write $\mathrm{BD}^*(E)>0$, if $\mathrm{D}_{\mathcal{F}}^*(E)>0$ corresponding to some classical  F{\o}lner sequence $\mathcal{F}=(F_n)_{n=1}^\infty$ in $(G,+,|\centerdot|)$.
We shall be concerned with measurable sets $E\subseteq G$ having $\mathrm{BD}^*(E)>0$.

The locally compact Hausdorff additive topological group $(G,+)$ is said to be \textit{amenable} if the so-called F{\o}lner condition holds; that is, for any compact set $K\subseteq G$ and any $\varepsilon>0$, there exists some compact set $F\subseteq G$ such that
$$\frac{|(K+F)\vartriangle F|}{|F|}<\varepsilon.$$
See, e.g., \cite{Pat}.
Each locally compact Hausdorff abelian topological group is amenable. Of course, any amenable group does not need to have a classical F{\o}lner sequence in it if no the $\sigma$-compact condition~\cite{Pat}. However, there always exists a $K$-F{\o}lner sequence in the sense of Definition~\ref{def2.1} for any compact subset $K$ of an amenable group $G$.

\subsection{Preliminary lemmas}
We can then obtain the following simple result from the definition of amenable group:

\begin{lem}\label{lem2.3}
If $(G,+)$ is an amenable and locally compact Hausdorff topological group, then for any compact set $K\subseteq G$ there exists a $K$-F{\o}lner sequence $(F_n)_{n=1}^\infty$ in $(G,+,|\centerdot|)$.
\end{lem}

\begin{proof}
Since $G$ is amenable and locally compact, then by \cite[Theorem~4.10]{Pat} it follows that for any compact set $K\subseteq G$ and $\varepsilon_n>0$, there exists some compact set $F_n\subseteq G$ such that
$$\frac{|(g+F_n)\vartriangle F_n|}{|F_n|}<\varepsilon_n\quad \forall g\in K.$$
Letting $\varepsilon_n\to0$ implies the desired result.
\end{proof}

This lemma enables us to choose an $F$-F{\o}lner sequence $\mathcal{F}=(F_n)_{n=1}^\infty$ in $(G,+,|\centerdot|)$ for any compact set $F\subseteq G$ in Theorems~\ref{thm3.1} and \ref{thm3.3} below.

It is well known that any discrete countable abelian group $G$ has $G$-F{\o}lner sequences. Although this is not the case for any discrete abelian group, yet we can obtain the following

\begin{lem}\label{lem2.4}
Let $(G,+)$ be a discrete abelian group. Then for any finite subset $F\subseteq G$ and any F{\o}lner net $(F_\theta)_{\theta\in\Theta}$ in $(G,+,|\centerdot|)$, there exists an $F$-F{\o}lner sequence $(F_n)_{n=1}^\infty$ in $(G,+,|\centerdot|)$ such that $(F_n)_{n=1}^\infty\subseteq(F_\theta)_{\theta\in\Theta}$.
\end{lem}

\begin{proof}
Let $(F_\theta)_{\theta\in\Theta}$ be a F{\o}lner net in $(G,+,|\centerdot|)$; that is, $(\Theta,\ge)$ is a directed index set and $F_\theta$ is a finite subset of $G$ such that
$$
\lim_{\theta\in\Theta}\frac{|(g+F_\theta)\vartriangle F_\theta|}{|F_\theta|}=0\quad \forall g\in G.
$$
Then for any $\epsilon>0$, there is some $\theta_\epsilon\in\Theta$ such that
$$
\frac{|(g+F_\theta)\vartriangle F_\theta|}{|F_\theta|}<\epsilon\quad \forall g\in F\textrm{ and }\theta\ge\theta_\epsilon.
$$
Letting $\epsilon\to0$ we can find a desired $F$-F{\o}lner sequence $(F_n)_{n=1}^\infty$ in $(G,+,|\centerdot|)$.

This thus proves Lemma~\ref{lem2.4}.
\end{proof}

Given any $K\subset G$, by $\langle K\rangle$ we denote the subgroup of $(G,+)$ spanned by $K$; that is,
\begin{gather*}
\langle K\rangle=\{k_1+\dotsm+k_n\,|\,n\ge1, k_i\in K\cup(-K)\},
\end{gather*}
where $-K=\{-k\,|\,k\in K\}$. The following is a simple observation.

\begin{lem}\label{lem2.5}
Let $(G,+)$ be an amenable locally compact Hausdorff topological group and $K\subseteq G$ any compact set. If $(F_n)_{n=1}^\infty$ is a $K$-F{\o}lner sequence in $(G,+,|\centerdot|)$, then it is also a $\langle K\rangle$-F{\o}lner sequence in $(G,+,|\centerdot|)$.
\end{lem}

\begin{proof}
Given any $g,g_1,g_2\in K$, by (\ref{eq1.1}) with $K$ in place of $F$, we have
\begin{equation*}\begin{split}
&\lim_{n\to\infty}\frac{|(g_2+g_1+F_n)\vartriangle F_n|}{|F_n|}\\
&\le\lim_{n\to\infty}\frac{|(g_2+g_1+F_n)\setminus F_n|+|F_n\setminus(g_2+g_1+F_n)|}{|F_n|}\\
&=\lim_{n\to\infty}\frac{|(g_2+g_1+F_n)\setminus(g_2+F_n)|+|(g_2+F_n)\setminus(g_2+g_1+F_n)|}{|F_n|}\\
&=0
\end{split}\end{equation*}
and
\begin{gather*}
\lim_{n\to\infty}\frac{|(-g+F_n)\vartriangle F_n|}{|F_n|}=\lim_{n\to\infty}\frac{|F_n\vartriangle(g+F_n)|}{|F_n|}=0.
\end{gather*}
Thus $(F_n)_{n=1}^\infty$ is a $\langle K\rangle$-F{\o}lner sequence in $(G,+,|\centerdot|)$.
\end{proof}

We notice here that, as a $\sigma$-compact topological group, $(\langle K\rangle,+)$ itself does not need to be amenable, since the subgroup $\langle K\rangle$ is not necessarily a closed subset of $G$. In addition, $(F_n\cap\langle K\rangle)_{n=1}^\infty$ does not need to be a F{\o}lner sequence in $\langle K\rangle$. In fact, $F_n\cap\langle K\rangle=\emptyset$ in the example constructed before. Moreover, it is possible that $E\cap\langle K\rangle=\emptyset$ in Definition~\ref{def2.2}.

\section{F\"{u}rstenberg correspondence principle}\label{sec2}
Let $(G,+)$ be a locally compact Hausdorff additive topological group with the zero element $\e$. Given any compact Hausdorff space $X$, we shall say that
\begin{gather*}
T\colon G\times X\rightarrow X\quad \textrm{or simply write}\quad G\curvearrowright_TX
\end{gather*}
is a Borel $G$-action on $X$ or $X$ is called a Borel $G$-space if it holds that
\begin{itemize}
\item $T_g\colon x\mapsto T(g,x)$ is a continuous selfmap of $X$, for each $g\in G$;
\item $T\colon (g,x)\mapsto T(g,x)$ is jointly Borel measurable;
\item $T_{\e}x=x$ for each $x\in X$ and $T_{g+h}=T_g\circ T_h$ for all $g,h\in G$.
\end{itemize}
In this section, we will consider F\"{u}rstenberg correspondence principles between configurations in subsets of $G$ and Borel $G$-space $X$ associated to $G$.

\medskip

The following Szemer\'{e}di-type theorem (Theorem~\ref{thm3.1}) is one of our main results, in the proof of which there are two ingredients in our F\"{u}rstenberg correspondence principle: (1) the compact Hausdorff $G$-space $X$ is not necessarily metrizable; and (2) although the topology of $(G,+,|\centerdot|)$ does not need to be discrete, yet to define the associated $G$-action we will employ the discrete topology that is not necessarily compatible with the fixed Haar measure $|\centerdot|$ on $G$.

\begin{thm}\label{thm3.1}
Let $(G,+)$ be a locally compact Hausdorff abelian topological group and $F\subseteq G$ a compact subset. If a measurable $E\subseteq G$ is of positive upper density corresponding to an $F$-F{\o}lner sequence, $\mathcal{F}=(F_n)_{n=1}^\infty$, in $(G,+,|\centerdot|)$, then for any $g_1^{},\dotsc,g_l^{}\in\langle F\rangle$,
$$
\mathrm{BD}_*\left(\{d\in\mathbb{Z}\,|\,\mathrm{D}_\mathcal{F}^*(\{u\in E\colon u+d\{g_1^{},\dotsc,g_l^{}\}\subseteq E\})>0\}\right)>0.
$$
Here $\mathrm{BD}_*$ denotes the lower Banach density of sets in $(\mathbb{Z},+,|\centerdot|_{\mathbb{Z}})$ and $\langle F\rangle$ stands for the subgroup of $G$ generated by $F$.
\end{thm}

\begin{proof}
By refining the $F$-F{\o}lner sequence $\mathcal{F}=(F_n)_{n=1}^\infty$ if necessary, we may assume
\begin{equation*}
\mathrm{D}_\mathcal{F}^*(E)=\lim_{n\to\infty}\frac{|E\cap F_n|}{|F_n|}.
\end{equation*}
Let $X=\prod_{g\in G}\{0,1\}$ be the Cartesian product endowed with the standard pointwise convergence topology. Then $X$ is a compact Hausdorff space. For any $x\in X$, we may identify it with the function $x(\centerdot)\colon g\mapsto x(g)$ from $G$ into the discrete space $\{0,1\}$. Note that $x(\centerdot)\in X$ is not necessarily a measurable function from $G$ to $\{0,1\}$ under the locally compact Hausdorff topology of $G$ and that the compact Hausdorff topology of $X$ is independent of the topology of $G$.

Let $\chi\in X$ be given by $\chi(g)=1$ if and only if $g\in E$. Since by hypothesis $E$ is $|\centerdot|$-measurable, hence $\chi(\centerdot)\colon G\rightarrow\{0,1\}$ is $|\centerdot|$-measurable under the locally compact Hausdorff topology of $G$.

Define the clopen cylinder set of $X$,
$[1]_{\e}=\{x\in X\,|\,x(\e)=1\}$, which is a compact $G_\delta$-set and so is a Baire set of $X$.
Then the characteristic function $1_{[1]_{\e}}(x)$ of $[1]_{\e}$ is a continuous function on $X$, i.e., $1_{[1]_{\e}}(x)\in C(X)$.

As in the usual F\"{u}rstenberg correspondence principle, \textit{under the discrete topology of $G$} we may now define a $G$-action on $X$ as follows:
$$
T\colon G\times X\rightarrow X\quad\textrm{or write}\quad G\curvearrowright_TX;\quad (g,x)\mapsto T_gx=x(\centerdot+g)
$$
where $x(\centerdot+g)\in X$ is defined by
\begin{gather*}
x(\centerdot+g)\colon t\mapsto x(t+g).
\end{gather*}
It should be noted that the continuity of $T_g\colon X\rightarrow X$ is obvious under the product topology of $X$. Thus, under the discrete topology of $G$, $G\curvearrowright_TX$ is a canonical $G$-action; in other words, $X$ is a Borel $G$-space.

By the Riesz representation theorem we can identify Baire measures on $X$ with positive functionals on $C(X)$. Using the refined $F$-F{\o}lner sequence $\mathcal{F}=(F_n)_{n=1}^\infty$ in $(G,+,|\centerdot|)$ we shall define a Baire probability on the product space $X$. Namely
$$
\mu_n(\varphi)=\frac{1}{|F_n|}\int_{F_n}\varphi(T_g\chi)dg\quad \forall \varphi\in C(X),
$$
noting that $\varphi(T_\centerdot\chi)\colon g\mapsto\varphi(T_g\chi)$ is $|\centerdot|$-measurable and $dg$-integrable from $F_n$ to $\{0,1\}$ under the topology of $(G,+,|\centerdot|)$ because of the measurability of $E$ and that the Haar measure $|\centerdot|$ is not defined by the newly introduced discrete topology of $G$.

Under the usual weak-$*$ topology of Baire probability measures on $X$, we can find a net $(\mu_\theta)_{\theta\in\Theta}$, which is a subnet of the Baire probability measure sequence $(\mu_n)_{n=1}^\infty$, such that
$$
\textrm{weakly-$*$ }\lim_{\theta\in\Theta}\mu_\theta=\mu
$$
for some Baire probability measure $\mu$ on $X$. By Lemma~\ref{lem2.5}, it is routine to check that $\mu$ is $T_g$-invariant for each $g\in\langle F\rangle$ (not for any $g\in G$).

Since $\mathrm{D}_\mathcal{F}^*(E)>0$, hence $\mu([1]_{\e})>0$. Indeed, by $1_{[1]_{\e}}(\centerdot)\in C(X)$ and $1_{[1]_{\e}}(T_g\chi)=1_E(g)$,
\begin{align*}
\mu([1]_{\e})&=\lim_{\theta\in\Theta}\frac{1}{|F_\theta|}\int_{F_\theta}1_{[1]_{\e}}(T_g\chi)dg\\
&=\lim_{\theta\in\Theta}\frac{|E\cap F_\theta|}{|F_\theta|}\\
&=\mathrm{D}_\mathcal{F}^*(E).
\end{align*}
Let $g_1^{},\dotsc,g_l^{}\in\langle F\rangle$ be arbitrarily given. Then by F\"{u}rstenberg's multiple recurrence theorem (cf.~\cite{FK,Fur}), it follows that
$$
D=\left\{d\in\mathbb{Z}\,|\,\mu\left([1]_{\e}\cap T_{g_1^{}}^{-d}[1]_{\e}\cap\dotsm\cap T_{g_l^{}}^{-d}[1]_{\e}\right)>0\right\}
$$
is of positive lower Banach density in $(\mathbb{Z},+,|\centerdot|_\mathbb{Z})$. Set $K=\{g_1^{},\dotsc,g_l^{}\}$. Next, given any $d\in D$ we set
\begin{gather*}
A=\{u\in E\,|\, u+dK\subseteq E\}\intertext{and} U=[1]_{\e}\cap T_{g_1^{}}^{-d}[1]_{\e}\cap\dotsm\cap T_{g_l^{}}^{-d}[1]_{\e}.
\end{gather*}
Then $U$ is clopen and so is a Baire set in $X$ for $T_g$ is continuous of $X$ to itself and hence
\begin{align*}
\mu(U)&=\lim_{\theta\in\Theta}\frac{1}{|F_\theta|}\int_{F_\theta}1_U(T_g\chi)dg\\
&\le \lim_{\theta\in\Theta}\frac{|A\cap F_\theta|}{|F_\theta|}\\
&\le \limsup_{n\to\infty}\frac{|A\cap F_n|}{|F_n|}\\
&=\mathrm{D}_\mathcal{F}^*(A).
\end{align*}
This proves Theorem~\ref{thm3.1}.
\end{proof}

Let $\mathcal{B}a(X)$ be the $\sigma$-algebra of all Baire subsets of $X$. It should be noted that for $G\curvearrowright_T(X,\mathcal{B}a(X),\mu)$ associated to $E$ in the proof of Theorem~\ref{thm3.1}, $X$ is never metrizable if $G$ is uncountable. Further since in our context there is no the ergodic decomposition theorem for $(X,\mathcal{B}a(X),\mu)$ is not (isomorphic to) a Polish probability space and so no quasi-generic point, hence the proof of $\mathrm{D}_\mathcal{F}^*(\{u\in E\,|\, u+d\{g_1^{},\dotsc,g_l^{}\}\subseteq E\})>0$ is of interest.

An interesting consequence of the above theorem is the following, which is a generalization of the classical Szemer\'{e}di theorem for $G=\mathbb{Z}$ due to E.~Szemer\'{e}di~\cite{Sze}, for $G=\mathbb{Z}^m$ due to F\"{u}rstenberg and Katznelson~\cite{FK}, and for $G=\mathbb{R}^m$ with the Euclidean metric topology due to H.~F\"{u}rstenberg~\cite[Theorem~7.17]{Fur}.

\begin{cor}\label{cor3.2}
Let $(G,+)$ be a second countable locally compact Hausdorff abelian topological group. If a measurable set $E\subseteq G$ is of positive upper Banach density, i.e., $\mathrm{BD}^*(E)>0$, then for any $g_1^{},\dotsc,g_l^{}\in G$
$$
\mathrm{BD}_*\left(\{d\in\mathbb{Z}\,|\,\mathrm{BD}^*(\{u\in E\colon u+d\{g_1^{},\dotsc,g_l^{}\}\subseteq E\})>0\}\right)>0.
$$
\end{cor}

\begin{proof}
Since every second countable locally compact Hausdorff group is $\sigma$-compact, hence $G$ has a classical F{\o}lner sequence. Then $\mathrm{BD}^*(E)$ makes sense and the statement follows at once from Theorem~\ref{thm3.1} with $F=G$.
\end{proof}

If we now utilize Bergelson-Leibman polynomial multiple recurrence theorem (cf.~\cite{BL96}) instead of F\"{u}rstenberg's multiple recurrence theorem in the proof of Theorem~\ref{thm3.1}, then we can easily obtain the following.

\begin{thm}\label{thm3.3}
Let $(G,+)$ be a locally compact Hausdorff abelian group and $F\subseteq G$ a compact subset. If a measurable set $E\subseteq G$ is of positive upper density corresponding to an $F$-F{\o}lner sequence $\mathcal{F}=(F_n)_{n=1}^\infty$ in $(G,+,|\centerdot|)$. Then for any $g_1^{},\dotsc,g_l^{}\in\langle F\rangle$
$$
\mathrm{BD}_*\left(\{d\in\mathbb{Z}\,|\,\mathrm{D}_\mathcal{F}^*(\{u\in E\colon u+\{p_1(d)g_1^{}, \dotsc, p_l(d)g_l^{}\}\subseteq E\})>0\}\right)>0
$$
for any $l$ polynomials $p_1(t), \dotsc, p_l(t)\in\mathbb{Z}[t]$ with $p_i(0)=0$ for $i=1,\dotsc,l$.
\end{thm}

Recall that for any discrete abelian group $(G,+)$, $E\subseteq G$ is of positive upper Banach density if and only if there exists a F{\o}lner net $(F_\theta)_{\theta\in\Theta}$ in $(G,+,|\centerdot|)$ such that
$$
\mathrm{BD}^*(E)=\lim_{\theta\in\Theta}\frac{|E\cap F_\theta|}{|F_\theta|}>0.
$$
Then by Theorem~\ref{thm3.3} together with Lemma~\ref{lem2.4}, we can obtain the following

\begin{cor}\label{cor3.4}
Let $(G,+)$ be a discrete abelian additive group and let $p_1(t), \dotsc, p_l(t)\in\mathbb{Z}[t]$ with $p_i(0)=0$ for $1\le i\le l$. If $E\subseteq G$ is such that $\mathrm{BD}^*(E)>0$ with F{\o}lner nets, then for any
$\{g_1^{},\dotsc,g_l^{}\}\subseteq G$ we have
\begin{gather*}
\mathrm{BD}_*\left(\{d\in\mathbb{Z}\,|\,\mathrm{BD}^*(\{u\in E\colon u+p_i(d)g_i^{}\in E\textrm{ for }i=1,\dotsc,l\})>0\}\right)>0.
\end{gather*}
\end{cor}

\begin{proof}
For any $\{g_1^{},\dotsc,g_l^{}\}\subseteq G$, let $F=\{g_1^{},\dotsc,g_l^{}\}$. By Lemma~\ref{lem2.4}, we can find some $F$-F{\o}lner sequence, say $\mathcal{F}=(F_n)_{n=1}^\infty$, in $(G,+,|\centerdot|)$ with $\mathrm{D}_\mathcal{F}^*(E)=\mathrm{BD}^*(E)$. Then the statement follows from Theorem~\ref{thm3.3}.
\end{proof}

We notice here that $G$ does not have any classical F{\o}lner sequence with the discrete topology when $G$ is uncountable and moreover the finite set $\{g_1^{}=a, g_2^{}=2a, \dotsc, g_l^{}=la\}$ in \cite{HS06} is a special configuration in Corollary~\ref{cor3.2}.
A special case of Theorem~\ref{thm3.3} is the following

\begin{cor}\label{cor3.5}
Let $(G,+)$ be an amenable group, $F\subset G$ a compact set and $E\subseteq G$ with $\mathrm{D}_\mathcal{F}^*(E)>0$ respecting to some $F$-F{\o}lner sequence $\mathcal{F}=(F_n)_{n=1}^\infty$ in $G$. Then for any
$g\in\langle G\rangle$, $l\in\mathbb{N}$ and $p_1(t), \dotsc, p_l(t)\in\mathbb{Z}[t]$ with $p_i(0)=0$ for $i=1,\dotsc,l$, there is some $d\in\mathbb{N}$ such that
$$
\mathrm{D}_{\mathcal{F}}^*(\{u\in E\colon u+\{p_1(d)g,2p_2(d)g,\dotsc,lp_l(d)g\}\subseteq E\})>0.
$$
\end{cor}

\begin{proof}
This follows obviously from applying Theorem~\ref{thm3.3} with the special case $g_1^{}=g, g_2^{}=2g, \dotsc, g_l^{}=lg$ and noting $T_{g_1^{}},\dotsc,T_{g_l^{}}$ are commuting.
\end{proof}

Finally we note that in Theorems~\ref{thm3.1} and \ref{thm3.3}, we cannot consider $(\langle F\rangle,+)$ as an independent abelian group, because the $F$-F{\o}lner sequence $(F_n)_{n=1}^\infty$ does only belong to $G$, not to $\langle F\rangle$, and moreover, $E$ we consider here may have a void intersection with $\langle F\rangle$ and $\mathrm{D}_\mathcal{F}^*(E)$ is associated to $(F_n)_{n=1}^\infty$ in $G$.


\subsection*{Acknowledgements}
The authors would like to thank the referee for his/her many constructive comments.

This work was partly supported by National Natural Science Foundation of China (Grant Nos. 11431012 and 11271183) and PAPD of Jiangsu Higher Education Institutions.

\end{document}